\documentclass[a4paper]{amsart}

\usepackage[mathcal]{euscript}
\usepackage{amssymb}
\usepackage[all]{xy}
\SelectTips{cm}{}
\usepackage{hyperref}

\theoremstyle{plain}
\newtheorem{thm}{Theorem}[subsection]
\newtheorem{lem}[thm]{Lemma}
\newtheorem{prop}[thm]{Proposition}
\newtheorem{cor}[thm]{Corollary}

\theoremstyle{definition}
\newtheorem{defn}[thm]{Definition}
\newtheorem{ex}[thm]{Example}

\theoremstyle{remark}
\newtheorem{rem}[thm]{Remark}

\numberwithin{equation}{section}

\def\to{\rightarrow}
\def\lto{\longrightarrow}
\def\Ab{\mathrm{Ab}}
\def\AA{\mathbb{A}}
\def\PP{\mathbb{P}}
\def\ZZ{\mathbb{Z}}
\def\str{\mathcal{O}}

\def\sfA{\mathsf{A}}
\def\sfD{\mathsf{D}}

\def\sfT{\mathsf{T}}

\def\sfI{\mathsf{I}}
\def\sfJ{\mathsf{J}}

\def\sfL{\mathsf{L}}

\def\mcA{\mathcal{A}}

\def\mcE{\mathcal{E}}

\def\mcU{\mathcal{U}}
\def\mcV{\mathcal{V}}

\def\mcX{\mathcal{X}}

\DeclareMathOperator{\End}{End}
\DeclareMathOperator{\Spec}{Spec}
\DeclareMathOperator{\Spc}{Spc}
\DeclareMathOperator{\supp}{supp}
\DeclareMathOperator{\Hom}{Hom}
\DeclareMathOperator{\RHom}{\mathbf{R}Hom}
\DeclareMathOperator{\sHom}{\mathcal{H}\!\!\;\mathit{om}}
\DeclareMathOperator{\Ext}{Ext}

\DeclareMathOperator{\Modu}{Mod}
\DeclareMathOperator{\modu}{mod}
\DeclareMathOperator{\Coh}{Coh}
\DeclareMathOperator{\QCoh}{QCoh}
\DeclareMathOperator{\add}{add}
\DeclareMathOperator{\Loc}{Loc}
\DeclareMathOperator{\Thick}{Thick}
\DeclareMathOperator{\Ker}{Ker}
\DeclareMathOperator{\Add}{Add}
\DeclareMathOperator{\Ind}{Ind}

\title{The derived category of the projective line}
\author{Henning Krause}
\address{Henning Krause, Fakult\"at f\"ur Mathematik, Universit\"at
  Bielefeld, 33501 Bielefeld, Germany.}
\email{hkrause@math.uni-bielefeld.de}
\author{Greg Stevenson}
\address{Greg Stevenson, School of Mathematics and Statistics,
University of Glasgow,
University Place,
Glasgow G12 8SQ, Scotland U.K.
}
\email{gregory.stevenson@glasgow.ac.uk}

\begin{document}

\begin{abstract}
  We examine the localizing subcategories of the derived category of quasi-coherent sheaves on the projective line over a
  field. We provide a complete classification of all such subcategories which arise as the kernel of a cohomological functor to a Grothendieck category.
\end{abstract}

\maketitle

\section{Introduction}

Ostensibly, this article is about the projective line over a field, but secretly it is an invitation to a discussion of some open questions in the study of derived categories. More specifically, we are thinking of localizing subcategories and to what extent one can hope for a complete classification. The case of affine schemes is by now quite well understood, having been settled by Neeman in his celebrated chromatic tower paper
\cite{NeeChro}. However, surprisingly little is known in the simplest
non-affine case, namely the projective line over a field. We seek to begin to rectify this state of affairs and to advertise this and similar problems. 

Let us start by recalling what is known. We write $\QCoh\PP^1_k$ for the category of quasi-coherent sheaves on
the projective line $\PP^1_k$ over a field $k$, and $\Coh\PP^1_k$
denotes the full subcategory of coherent sheaves. There is a complete
description of the objects of $\Coh\PP^1_k$, due to Grothendieck \cite{Gr1957}. The
localizing subcategories of $\QCoh\PP^1_k$ are known by work of Gabriel \cite{Ga1962}, and are parametrized by specialization closed collections of points of $\PP^1_k$. When one passes to the derived
category $\sfD(\QCoh \PP^1)$, the situation becomes considerably more complicated. Several new localizations appear as a result of the fact that one can no longer non-trivially talk about subobjects and it remains a challenge to provide a complete classification of localizing subcategories.

An enticing aspect of this problem is that it not only represents the first stumbling block for those coming from algebraic geometry, but also for the representation theorists. There is an equivalence of triangulated categories 
\[\sfD(\QCoh \PP^1)\stackrel{\sim}\longrightarrow\sfD(\Modu A)\]
where $\Modu A$ denotes the module category of
\[A=\left[\begin{smallmatrix}k&k^2\\
      0&k\end{smallmatrix}\right].\] The algebra $A$ is isomorphic to
the path algebra of the Kronecker quiver
$\cdot\,
\genfrac{}{}{0pt}{}{\raisebox{-1.75pt}{$\longrightarrow$}}{\raisebox{1.75pt}{$\longrightarrow$}}\,
\cdot$ and is known to be of tame representation type. The ring $A$ is
one of the simplest non-representation finite algebras and so
understanding its derived category is also a key question from the
point of view of representation theory. Of particular note, it is
known by work of Ringel \cite{Ri1979, Ri1999} that $\Modu A$, the
category of all representations, is wild and so it is very natural to
ask if, as in the case of commutative noetherian rings, localizations
can nonetheless be classified.

In this article we make a contribution toward this challenge in two different ways. First of all, one of the main points of this work is to highlight this problem, provide some appropriate background, and set out what is known. To this end the first part of the article discusses the various types of localization one might consider in a compactly generated triangulated category and sketches the localizations of $\sfD(\QCoh \PP^1)$ which are known.

Our second contribution is to provide new perspective and new tools. The main new result is that the subcategories we understand admit a natural intrinsic characterization: it is shown in Theorem~\ref{thm:classification} that they are precisely the cohomological ones. In the final section we provide a discussion of the various restrictions that would have to be met by a non-cohomological localizing subcategory. Here our main results are that such subcategories come in $\ZZ$-families and consist of objects with full support on $\PP^1$.

\section{Preliminaries}

This section contains some background on localizations, localizing
subcategories, purity, and the projective line. Also it serves to fix
notation and may be safely skipped, especially by experts, and
referred back to as needed.

\subsection{Localizing subcategories and localizations}

Let $\sfT$ be a triangulated category with all small coproducts and
products. The case we have in mind is that $\sfT$ is either
well-generated or compactly generated.

\begin{defn} A full subcategory $\sfL$ of $\sfT$ is \emph{localizing}
if it is closed under suspensions, cones, and coproducts. This is
equivalent to saying that $\sfL$ is a coproduct closed triangulated
subcategory of $\sfT$.
\end{defn}

\begin{rem} It is a consequence of closure under (countable)
coproducts that $\sfL$ is closed under direct summands
and hence \emph{thick} (which means closed under finite sums, summands, suspensions, and cones).
\end{rem}

Given a collection of objects $\mcX$ of $\sfT$ we denote by
$\Loc(\mcX)$ the localizing subcategory generated by $\mcX$, i.e.\
the smallest localizing subcategory of $\sfT$ containing $\mcX$. The
collection of localizing subcategories is partially ordered by
inclusion, and forms a lattice (with the caveat it might not be a set)
with meet given by intersection. 

We next present the most basic reasonableness condition a localizing
subcategory can satisfy.

\begin{defn} A localizing subcategory $\sfL$ of $\sfT$ is said to be
\emph{strictly localizing} if the inclusion $i_*\colon \sfL \to \sfT$
admits a right adjoint $i^!$, i.e.\ if $\sfL$ is coreflective.
\end{defn}

Some remarks on this are in order. First of all, it follows that $i^!$
is a Verdier quotient, and that there is a localization sequence
\begin{displaymath} \xymatrix{ \sfL \ar[r]<0.5ex>^-{i_*}
\ar@{<-}[r]<-0.5ex>_-{i^!} & \sfT \ar[r]<0.5ex>^-{j^*}
\ar@{<-}[r]<-0.5ex>_-{j_*} & \sfT/\sfL }
\end{displaymath} inducing a canonical equivalence
\begin{displaymath} \sfL^\perp := \{X\in \sfT \mid \Hom(\sfL,X) =
0\} \xrightarrow{\sim} \sfT/\sfL.
\end{displaymath} Next we note that in nature localizing
subcategories tend to be strictly localizing. This is, almost uniformly, a
consequence of Brown representability; if $\sfT$ is well-generated and
$\sfL$ has a generating set of objects then $\sfL$ is strictly
localizing.

Now let us return to the localization sequence above. From it we
obtain two endofunctors of $\sfT$, namely
\begin{displaymath} i_*i^! \quad \text{and} \quad j_*j^*
\end{displaymath} which we refer to as the associated \emph{acyclization} and
\emph{localization} respectively. They come together with a counit and a unit
which endow them with the structure of an idempotent comonoid and
monoid respectively. The localization (or acyclization) is equivalent
information to $\sfL$. One can give an abstract definition of a
localization functor on $\sfT$ (or in fact any category) and then work
backward from such a functor to a strictly localizing
subcategory. Further details can be found in \cite{KrLoc}. We will use
the language of (strictly) localizing subcategories and localizations
interchangeably.

\subsection{Purity}

Let $\sfT$ be a compactly generated triangulated category and let
$\sfT^c$ denote the thick subcategory of compact objects. We denote by
$\Modu \sfT^c$ the Grothendieck category of modules over $\sfT^c$,
i.e.\ the category of contravariant additive functors
$\sfT^c \to \Ab$. There is a \emph{restricted Yoneda functor}
\begin{equation}\label{eq:yoneda}
H\colon \sfT \lto \Modu \sfT^c \text{ defined
by }  HX = \Hom(-,X)|_{\sfT^c},
\end{equation} 
which is cohomological, conservative, and preserves
both products and coproducts.

\begin{defn} A morphism $f\colon X\to Y$ in $\sfT$ is a
\emph{pure-monomorphism} (resp.\ \emph{pure-epimorphism}) if $Hf$ is a
monomorphism (resp.\ epimorphism).

An object $I\in \sfT$ is \emph{pure-injective} if every
pure-monomorphism $I\to X$ is split, i.e.\ it is injective with respect to pure-monomorphisms.
\end{defn}

It is clear from the definition that if $I\in \sfT$ with $HI$
injective then $I$ is pure-injective. It turns out that the converse
is true and so $I$ is pure-injective if and only if $HI$ is
injective. Moreover, Brown representability allows one to lift any
injective object of $\Modu \sfT^c$ uniquely to $\sfT$ and thus one
obtains an equivalence of categories
\begin{displaymath} \{\text{pure-injectives in}\;\sfT\} \xrightarrow{\sim}
\{\text{injectives in}\;\Modu \sfT^c\}.
\end{displaymath}

Further details on purity, together with proofs and references for
the above facts can be found, for instance, in \cite{Kr2000}.

\subsection{The projective line}\label{prelim:P1}

Throughout we will work over a fixed base field $k$ which will be
supressed from the notation. For instance, $\PP^1$ denotes the
projective line $\PP^1_k$ over $k$. We will denote by $\eta$ the
generic point of $\PP^1$. The points of $\PP^1$ that are different
from $\eta$ are closed. A subset $\mcV\subseteq \PP^1$ is
specialization closed if it is the union of the closures of its points. In our situation this just says that $\mcV$ is specialization closed if $\eta\in\mcV$ implies  $\mcV=\PP^1$.

As usual $\QCoh \PP^1$ is the Grothendieck
category of quasi-coherent sheaves on $\PP^1$ and $\Coh \PP^1$ is the
full abelian subcategory of coherent sheaves.

We use standard notation for the usual `distinguished' objects of
$\QCoh \PP^1$. The $i$th twisting sheaf is denoted $\str(i)$ and for a
point $x\in \PP^1$ we let $k(x)$ denote the residue field at $x$. In
particular, $k(\eta)$ is the sheaf of rational functions on
$\PP^1$. For an object $X\in \sfD(\QCoh \PP^1)$ or a localizing
subcategory $\sfL$ we will often write $X(i)$ and $\sfL(i)$ for
$X\otimes \str(i)$ and $\sfL\otimes \str(i)$ respectively.

All functors, unless explicitly mentioned otherwise, are derived. In
particular, $\otimes$ denotes the left derived tensor product of
quasi-coherent sheaves and $\sHom$ the right derived functor of the
internal hom in $\QCoh \PP^1$.

For an object $X\in \sfD(\QCoh \PP^1)$ we set
\begin{displaymath} \supp X = \{x\in \PP^1\;\vert\; k(x)\otimes X \neq
0\}.
\end{displaymath} This agrees with the notion of support one gets as
in \cite{BF2011} by allowing $\sfD(\QCoh \PP^1)$ to act on itself; the
localizing subcategories generated by $k(x)$ and $\Gamma_x\str$
coincide.

\begin{rem}
  Let $\sfA$ be a hereditary abelian category, for example
  $\QCoh \PP^1$. Then $\Ext^n(X,Y)$ vanishes for all $n>1$ and
  therefore every object of the derived category $\sfD(\sfA)$
  decomposes into complexes that are concentrated in a single
  degree. It follows that the functor $H^0\colon\sfD(\sfA)\to\sfA$
  induces a bijection between the localizing subcategories of
  $\sfD(\sfA)$ and the full subcategories of $\sfA$ that are closed
  under kernels, cokernels, extensions, and coproducts.
\end{rem}

\section{Types of localization}

In this section we give a further review of the notions of localization, or
equivalently localizing subcategory, that naturally arise and that we
treat in this article. These come in various strengths and what is
known in general varies accordingly. We take advantage of this review
to give a whirlwind tour of certain aspects of the subject and to
expose some technical results that are absent from the literature.

Unless otherwise specified we will denote by $\sfT$ a compactly
generated triangulated category. One also can, and should, consider
the well-generated case and it arises naturally even when one starts
with a compactly generated category. However, our focus will,
eventually, be on those categories controlled by pure-injectives which
more or less binds us to the compactly generated case.

\subsection{Smashing localizations}

In this section we make some brief recollections on the most well
understood class of localizing subcategories.

\begin{defn} A localizing subcategory $\sfL$ of $\sfT$ is
\emph{smashing} if it is strictly localizing and satisfies one, and
hence all, of the following equivalent conditions:
\begin{itemize}
\item the subcategory $\sfL^\perp$ is localizing;
\item the corresponding localization functor preserves coproducts,
i.e.\ the right adjoint to $\sfT \to \sfT/\sfL$ preserves coproducts;
\item the quotient functor $\sfT \to \sfT/\sfL$ preserves compactness;
\item the corresponding acyclization functor preserves coproducts,
i.e.\ the right adjoint to $\sfL \to \sfT$ preserves coproducts.
\end{itemize}
\end{defn}

The smashing subcategories always form a set. Amongst the smashing subcategories there is a potentially smaller
distinguished set of localizing subcategories. Unfortunately, there is
not a standard way to refer to such categories; the snappy
nomenclature only exists for the corresponding localizations.

\begin{defn} A localization is \emph{finite} if its kernel is
generated by objects of $\sfT^c$, i.e.\ the corresponding localizing
subcategory is generated by objects which are compact in $\sfT$.
\end{defn}

If $\sfL$ is the kernel of a finite localization then it is
smashing. It is also compactly generated, although there are in
general many localizing subcategories of $\sfT$ which are, as abstract
triangulated categories, compactly generated but are not generated by
objects compact in $\sfT$.

The \emph{smashing conjecture} for $\sfT$ asserts that every smashing
localization is a finite localization. This is true in many
situations, for instance it holds for the derived category
$\sfD(\Modu A)$ of a ring $A$ when it is commutative noetherian
\cite{NeeChro} or hereditary \cite{KS2010}. On the other hand it is
known to fail for certain rings (see for instance \cite{Keller}) and
is open in many cases of interest, for example the stable homotopy
category.

\subsection{Cohomological localizations}

We now come to the next types of localizing subcategories in our
hierarchy, which are defined by certain orthogonality conditions. This
gives a significantly weaker hierarchy of notions than being smashing. 

First a couple of reminders. An abelian category $\sfA$ is said to be
(AB5) if it is cocomplete and filtered colimits are exact. If in
addition $\sfA$ has a generator then it is a \emph{Grothendieck
  category}. An additive functor $H\colon \sfT\to \sfA$ is
\emph{cohomological} if it sends triangles to long exact sequences
i.e.\ given a triangle
\begin{displaymath} \xymatrix{ X \ar[r]^-f & Y \ar[r]^-g & Z
\ar[r]^-{h} & \Sigma X }
\end{displaymath} the sequence
\begin{displaymath} \xymatrix@C=1.7em{ \cdots \ar[r] & H(\Sigma^{-1}Z)
\ar[r]^-{\Sigma^{-1}h} & H(X) \ar[r]^-{H(f)} & H(Y) \ar[r]^-{H(g)} &
H(Z) \ar[r]^-{H(h)} & H(\Sigma X) \ar[r] & \cdots }
\end{displaymath} is exact in $\sfA$.

\begin{defn}\label{defn:cohom} A localizing subcategory
$\sfL\subseteq\sfT$ is \emph{cohomological} if there exists a
cohomological functor $H\colon\sfT\rightarrow\sfA$ into an (AB5)
abelian category such that $H$ preserves all coproducts and
\begin{displaymath} \sfL =\{X\in\sfT\mid H(\Sigma ^n X)=0\text{ for
all }n\in\mathbb Z\},
\end{displaymath} that is $\sfL$ is the kernel of $H^*$.
\end{defn}

We can extend this definition to an analogue for arbitrary regular
cardinals, with Definition~\ref{defn:cohom} being the $\aleph_0$ or
`base' case. The idea is to relax the exactness condition on the
target abelian category. This requires a little terminological
preparation.

Let $\sfJ$ be a small category and $\alpha$ a regular cardinal. We say
that $\sfJ$ is $\alpha$\emph{-filtered} if for every category $\sfI$
with $\vert I \vert < \alpha$, i.e.\ $\sfI$ has fewer than $\alpha$
arrows, every functor $F\colon \sfI \to \sfJ$ has a cocone. For
instance, this implies that any collection of fewer than $\alpha$
objects of $\sfJ$ has an upper bound and any collection of fewer than
$\alpha$ parallel arrows has a weak coequalizer. If $\alpha =
\aleph_0$ we just get the usual notion of a filtered category.

Let $\sfA$ be an abelian category. We say it satisfies (AB$5^\alpha$)
if it is cocomplete and has exact $\alpha$-filtered colimits.

\begin{defn} A localizing subcategory $\sfL\subseteq \sfT$ is
$\alpha$\emph{-cohomological} if there exists an (AB$5^\alpha$)
abelian category $\sfA$ and a coproduct preserving cohomological
functor $H\colon \sfT \to \sfA$ such that
\begin{displaymath} \sfL =\{X\in\sfT\mid H(\Sigma ^n X)=0\text{ for
all }n\in\mathbb Z\},
\end{displaymath} that is $\sfL$ is the kernel of $H^*$. 
\end{defn}

If $\sfL$ is $\alpha$-cohomological then it is clearly
$\beta$-cohomological for all $\beta \geq \alpha$.

\begin{rem} An $\aleph_0$-cohomological localizing subcategory is just
a cohomological localizing subcategory. We will usually stick to the
shorter terminology for the sake of brevity and to avoid a
proliferation of $\aleph$'s. 
\end{rem}

We now make a few observations on $\alpha$-cohomological localizing
subcategories and then make some further remarks on the case $\alpha =
\aleph_0$.

\begin{lem} Smashing subcategories are cohomological.
\end{lem}
\begin{proof} Suppose $\sfL$ is smashing. Then $\sfT/\sfL$ is
compactly generated and for $H$ we can take the composite
\begin{displaymath} \sfT \lto \sfT/\sfL \lto \Modu (\sfT/\sfL)^c
\end{displaymath} where the latter functor is the restricted Yoneda
functor \eqref{eq:yoneda}.
\end{proof}

\begin{thm} Let $\sfL$ be an $\alpha$-cohomological localizing
subcategory. Then $\sfL$ is generated by a set of objects and so it
is, in particular, strictly localizing.
\end{thm}
\begin{proof} This follows by applying \cite[Theorem~7.1.1]{KrLoc}
and then \cite[Theorem~7.4.1]{KrLoc}.
\end{proof}

\begin{cor}\label{cor:set=cohom} A localizing subcategory $\sfL$ is
generated by a set of objects of $\sfT$ if and only if there exists an
$\alpha$ such that $\sfL$ is $\alpha$-cohomological.
\end{cor}
\begin{proof} We have just seen that an $\alpha$-cohomological
localizing subcategory has a generating set. On the other hand if
$\sfL$ is generated by a set of objects then $\sfL$ is well-generated,
and so is strictly localizing, and the quotient $\sfT/\sfL$ is also
well-generated (see \cite[Theorem~7.2.1]{KrLoc}). One can then
compose the quotient $\sfT \to \sfT/\sfL$ with the universal functor
from $\sfT/\sfL$ to an (AB$5^\alpha$) abelian category, for a
sufficiently large $\alpha$, to get the required cohomological
functor.
\end{proof}

Let us now restrict to cohomological localizations and make the
connection to purity in triangulated categories.

\begin{prop}\label{prop:cohomological} A localizing subcategory
$\sfL\subseteq\sfT$ is cohomological if and only if there is a suspension stable
collection of pure-injective objects $(Y_i)_{i\in I}$ in $\sfT$ such
that $\sfL=\{X\in\sfT\mid \Hom(X,Y_i)=0\text{ for all }i\in I\}$.
\end{prop}
\begin{proof} 
Recall from (\ref{eq:yoneda}) the restricted Yoneda functor which we denote by $H_\sfT$, for clarity, for the duration of the proof. This functor identifies the full subcategory of pure-injective objects in $\sfT$ with the full subcategory of
injective objects in $\Modu\sfT^c$ as noted earlier (see \cite[Corollary~1.9]{Kr2000} for details).

A cohomological functor $H\colon\sfT\rightarrow\sfA$ that preserves
coproducts admits a factorisation $H=\bar H\circ H_\sfT$ such that $\bar
H\colon\Modu\sfT^c\rightarrow\sfA$ is exact and preserves coproducts;
see \cite[Proposition 2.3]{Kr2000}.  The full subcategory $\Ker\bar
H=\{M\in\Modu\sfT^c\mid \bar H(M)=0\}$ is a localizing subcategory, so
of the form $\{M\in\Modu\sfT^c\mid \Hom(M,N_i)=0 \text{ for all }i\in
I\}$ for a collection of injective objects $(N_i)_{i\in I}$ in
$\Modu\sfT^c$. Now choose pure-injective objects $(Y_i)_{i\in I}$ in
$\sfT$ such that $H_\sfT(Y_i)\cong N_i$ for all $i\in I$.
\end{proof}

\subsection{When things are sets}

As has been alluded to in the previous sections, it is a
significant subtlety that one does not usually know the class of all
localizing subcategories forms a set. In fact there is no example
where one knows that there are a set of localizing subcategories by
`abstract means'; all of the examples come from classification
results.

If one does know there are a set of localizing subcategories then life
is much easier. The purpose of this section is to give some indication
of this, and record some other simple observations. Everything here
should be known to experts, but these observations have not yet found
a home in the literature.

Let $\sfT$ be a well-generated triangulated category.

\begin{lem}\label{lem:locset} If the localising subcategories of $\sfT$ form a set then every
  localizing subcategory is generated by a set of objects (and hence
  by a single object).
\end{lem}
\begin{proof} Suppose, for a contradiction, that $\sfL$ is a
localizing subcategory of $\sfT$ which is not generated by a set of
objects. We define a proper chain of proper localizing subcategories
\begin{displaymath} \sfL_0 \subsetneq \sfL_1 \subsetneq \cdots
\subsetneq \sfL_\alpha \subsetneq \sfL_{\alpha+1} \subsetneq \cdots
\subsetneq \sfL,
\end{displaymath} each of which is generated by a set of objects, by
transfinite induction. For the base case pick any object $X_0$ of
$\sfL$ and set $\sfL_0 = \Loc(X_0)$. This is evidently generated by a
set of objects, namely $\{X_0\}$. By assumption $\sfL$ is not
generated by a set of objects so $\sfL_0 \subsetneq \sfL$. Suppose we
have defined a proper localizing subcategory $\sfL_\alpha$ of $\sfL$
which is generated by a set of objects. Since $\sfL_\alpha$ is proper
we may pick an object $X_{\alpha+1}$ in $\sfL$ but not in
$\sfL_\alpha$ and set
\begin{displaymath} \sfL_{\alpha+1} = \Loc(\sfL_\alpha, X_{\alpha+1})
\supsetneq \sfL_\alpha.
\end{displaymath} This is clearly still generated by a set of objects
and hence is still a proper subcategory of $\sfL$. For a limit ordinal
$\lambda$ we set
\begin{displaymath} \sfL_{\lambda} = \Loc(\sfL_\kappa \mid \kappa <
  \lambda).
\end{displaymath} Again this is generated by a set of objects (and so
is still not all of $\sfL$).

This gives an ordinal indexed chain of distinct localizing
subcategories of $\sfT$. However, this is absurd since the collection
of ordinals is not a set and so cannot be embedded into the set of all
localising subcategories of $\sfT$. Hence $\sfL$ must have a generating set (i.e.\ the above
construction must terminate).
\end{proof}

\begin{rem} 
The above argument does not use that $\sfT$ is
well-generated.
\end{rem}

One then deduces that all localizations are cohomological for an
appropriate cardinal.

\begin{lem} If the localising subcategories of $\sfT$ form a set then
  every localizing subcategory of $\sfT$  is $\alpha$-cohomological for
  some regular cardinal $\alpha$.
\end{lem}
\begin{proof} By the previous lemma the hypothesis imply that every
localizing subcategory of $\sfT$ is generated by a set of objects. It
then follows from Corollary~\ref{cor:set=cohom} that they are all
cohomological.
\end{proof}

One can, to some extent, also work in the other direction.

\begin{lem} If the collection
  \begin{displaymath} \bigcup_{\alpha\in \mathsf{Card}} \{\sfL \mid
    \sfL \text{ is $\alpha$-cohomological}\}
\end{displaymath} forms a set then the collection of all localising subcategories of
$\sfT$  also forms a set.
\end{lem}
\begin{proof} 
We have seen in Corollary~\ref{cor:set=cohom} that being
$\alpha$-cohomological for some $\alpha$ is the same as being
generated by a set of objects. Thus the hypothesis asserts that there
are a set of localizing subcategories which have generating sets. From
this perspective it is clear we can pick a regular cardinal $\kappa$
such that every localizing subcategory of $\sfT$ which is generated by
a set is generated by $\kappa$-compact objects. Moreover, since the
union in the statement of the lemma is both a set and indexed by a
class, we conclude that the chain stabilises and so, taking $\kappa$
larger if necessary, we may also assume every $\alpha$-cohomological
localizing subcategory of $\sfT$ is $\kappa$-cohomological.

If the localising subcategories of $\sfT$ do not form a set then, as
there are a set of $\kappa$-cohomological localizing subcategories,
there must be a localizing subcategory $\sfL$ which is not generated
by a set of objects. In particular
\begin{displaymath} \sfL' = \Loc(\sfL \cap \sfT^\kappa) \subsetneq
\sfL.
\end{displaymath} But this is nonsense. Since $\sfL'$ is a proper
localizing subcategory of $\sfL$ we can find some object $X$ in $\sfL$
but not in $\sfL'$ and consider $\sfL'' = \Loc(\sfL', X)$. Clearly
$\sfL''$ is still contained in $\sfL$, it properly contains $\sfL'$,
it is generated by a set and hence $\kappa$-cohomological, and it
contains the $\kappa$-compact objects of $\sfL$. These are not
compatible statements: we have assumed $\kappa$ large enough so that
$\sfL''$ must be generated by the $\kappa$-compact objects it contains
but this contradicts $\sfL' \subsetneq \sfL''$.
\end{proof}

\section{Cohomological localizations for the projective line}

We now turn to the example we have in mind, namely $\sfD(\QCoh\PP^1)$
the unbounded derived category of quasi-coherent sheaves on
$\PP^1$. We first describe the thick subcategories of
$\sfD^\mathrm{b}(\Coh \PP^1)$, the bounded derived category of
coherent sheaves on $\PP^1$. We then recall the classifications of
smashing subcategories and of tensor ideals in $\sfD(\QCoh\PP^1)$. Finally,
we classify the ($\aleph_0$-)cohomological localizing
subcategories\textemdash{}there are no surprises and they are exactly
the ones which have been understood for some time.

It is of course possible that there are $\alpha$-cohomological
localizing subcategories for $\alpha > \aleph_0$ which we are not
aware of. It is in some sense tempting to guess that this is not the
case, i.e.\ that our list is already a complete list of localizing
subcategories, but there is no real evidence for this. We close by
making some remarks on the hurdles that such an `exotic' localization
would have to clear.

Before getting on with this let us recapitulate the connection with representation theory. By a result of Beilinson \cite{Be1978} there is a tilting object $T\in\Coh \PP^1$ which induces an exact equivalence
\[\RHom(T,-)\colon\sfD(\QCoh \PP^1)\stackrel{\sim}\longrightarrow\sfD(\Modu A)\]
where $\Modu A$ denotes the module category of
\[A=\End(T)\cong\left[\begin{smallmatrix}k&k^2\\
0&k\end{smallmatrix}\right].\] Note that $A$ is isomorphic to the path
algebra of the Kronecker quiver $\cdot\,
\genfrac{}{}{0pt}{}{\raisebox{-1.75pt}{$\longrightarrow$}}{\raisebox{1.75pt}{$\longrightarrow$}}\,
\cdot$ and this algebra is known to be of tame representation type. In
fact, the representation theory of this algebra amounts to the
classification of pairs of $k$-linear maps, up to simultaneous
conjugation. The finite dimensional representations were already known to Kronecker \cite{Kr1890}.

\subsection{Thick subcategories of the bounded derived category}

The structure of the lattice of thick subcategories of
$\sfD^\mathrm{b}(\Coh \PP^1)$, which we recall in this section, has
been known for some time; it can be computed by hand using the fact
that $\Coh \PP^1$ is tame and hereditary.

The structure of the coherent sheaves on $\PP^1$ is well known: there
is a $\ZZ$-indexed family of indecomposable vector bundles and a
1-parameter family of torsion sheaves for each point on $\PP^1$.

For each $i\in \ZZ$ one has a thick subcategory
\begin{displaymath} \Thick(\str(i)) = \add(\Sigma^j\str(i)\mid
j\in \ZZ) \cong \sfD^\mathrm{b}(k)
\end{displaymath} where the identifications follow from the
computation of the cohomology of $\PP^1$. These are the only proper
non-trivial thick subcategories which are generated by vector bundles
and are also the only thick subcategories which are not tensor ideals.
Thus we have a lattice isomorphism
\begin{displaymath} 
\{\text{thick subcategories of }\sfD^\mathrm{b}(\Coh
\PP^1) \text{ generated by vector bundles}\}\xrightarrow{\sim}  \ZZ
\end{displaymath}
where $\ZZ$ denotes the lattice given by the following Hasse diagram:
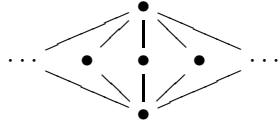
\[
  \xymatrix@=1em{
    &&\bullet\ar@{-}[dll]\ar@{-}[dl]\ar@{-}[dr]\ar@{-}[drr]\ar@{-}[d]\\
    \cdots\ar@{-}[drr]&\bullet\ar@{-}[dr]&\bullet\ar@{-}[d]&\bullet\ar@{-}[dl]&\cdots\ar@{-}[dll]\\
    &&\bullet }
\] 

This is a special case of a general result because the indecomposable
vector bundles are precisely the exceptional objects of
$\sfD^\mathrm{b}(\Coh \PP^1)$. For any hereditary artin algebra $A$
the thick subcategories of $\sfD^\mathrm{b}(\modu A)$ that are
generated by exceptional objects form a poset which is isomorphic to
the poset of non-crossing partitions given by the Weyl group $W(A)$;
see \cite{HK2016, IT2009}. Note that $W(A)$ is an affine Coxeter
group of type $\tilde A_1$ for the Kronecker algebra $A=\left[\begin{smallmatrix}k&k^2\\
    0&k\end{smallmatrix}\right]$, keeping in mind the derived
equivalence
\[\sfD^\mathrm{b}(\Coh
  \PP^1)\xrightarrow{\sim} \sfD^\mathrm{b}(\modu A).\]

The thick tensor ideals are classified by
$\Spc \sfD^\mathrm{b}(\Coh \PP^1)\cong \PP^1$, where the space
$\Spc \sfD^\mathrm{b}(\Coh \PP^1)$ is meant in the sense of Balmer
\cite{BaSpec}, and its computation is a special case of a general result of
Thomason \cite{Th1997}. What all this boils down to is that for any
set of closed points $\mcV$ of $\PP^1$ there is a thick tensor ideal
\begin{displaymath} \sfD_\mcV^\mathrm{b}(\Coh \PP^1) := \{E \mid\supp
  E\subseteq \mcV\} = \Thick(k(x)\mid x\in \mcV)
\end{displaymath} consisting of complexes of torsion sheaves supported on
$\mcV$. Moreover, together with $0$ and $\sfD^\mathrm{b}(\Coh
\PP^1)$ this is a complete list of thick tensor ideals. One can make
this uniform by considering subsets of $\PP^1$ which are
specialization closed. In this language, by extending the above notation to allow 
$\sfD_\varnothing^\mathrm{b}(\Coh \PP^1) = 0$ and
$\sfD_{\PP^1}^\mathrm{b}(\Coh \PP^1) = \sfD^\mathrm{b}(\Coh \PP^1)$,
we have a lattice isomorphism
\begin{displaymath} 
\{\text{thick tensor ideals of }\sfD^\mathrm{b}(\Coh \PP^1)\}\xrightarrow{\sim}
\{\text{spc subsets of }\PP^1\},
\end{displaymath}
where `spc' is an abbreviation for `specialization closed', which is given by
\begin{displaymath}
\sfI \mapsto \supp \sfI = \bigcup_{E\in \sfI}\supp E \quad\text{and}\quad \mcV \mapsto \sfD_\mcV^\mathrm{b}(\Coh \PP^1)
\end{displaymath}
for $\sfI$ a thick tensor ideal and $\mcV$ a specialization closed subset.

We know every object of $\sfD^\mathrm{b}(\Coh \PP^1)$ is a direct sum
of shifts of line bundles and torsion sheaves and so one can readily
combine these classifications to obtain a lattice
isomorphism\footnote{Let $L',L''$ be a pair of lattices with smallest
  elements $0',0''$ and greatest elements $1',1''$. Then
  $L'\amalg L''$ denotes the new lattice which is obtained from the
  disjoint union $L'\cup L''$ (viewed as sum of posets) by identifying
  $0'=0''$ and $1'=1''$.}
\begin{displaymath} 
\{\text{thick subcategories of }\sfD^\mathrm{b}(\Coh \PP^1)\}\xrightarrow{\sim}
\{\text{spc subsets of }\PP^1\} \amalg \ZZ.
\end{displaymath}

The verification that the evident bijection is indeed a lattice map as claimed is elementary: the twisting sheaves
are supported everywhere so are not contained in any proper ideal, and
any twisting sheaf and a torsion sheaf, or any pair of distinct
twisting sheaves, generate the category. Thus for $i\neq j$ and
$\mathcal{V}$ proper non-empty and specialization closed in $\PP^1$ we
have
\begin{displaymath} \sfD_\mathcal{V}^\mathrm{b}(\Coh \PP^1)\vee
\Thick(\str(i)) = \sfD^\mathrm{b}(\Coh \PP^1) = \Thick(\str(i)) \vee
\Thick(\str(j))
\end{displaymath} and
\begin{displaymath} \sfD_\mathcal{V}^\mathrm{b}(\Coh \PP^1) \wedge
\Thick(\str(i)) = 0 = \Thick(\str(i)) \wedge \Thick(\str(j)).
\end{displaymath}

\subsection{Ideals and smashing subcategories}\label{sec:smashing}

We now describe the localizing subcategories that one easily
constructs from our understanding of the compact objects
$\sfD^\mathrm{b}(\Coh \PP^1)$ in $\sfD(\QCoh \PP^1)$.

By \cite{KS2010} the smashing conjecture holds for $\sfD(\QCoh
\PP^1)$ (our computations will also essentially give a direct proof of this fact). Thus the finite localizations one obtains by inflating the
thick subcategories listed above exhaust the smashing localizations
i.e.\
\begin{displaymath}
\{\text{thick subcategories of }\sfD^\mathrm{b}(\Coh \PP^1)\}\xrightarrow{\sim} \{\text{smashing subcategories of }\sfD(\QCoh \PP^1)\}
\end{displaymath}

The localizing ideals are also understood. Again this is known more
generally (there is such a classification for any locally noetherian
scheme, see \cite{AJS}) but can easily be computed by hand for
$\PP^1$. The precise statement is that there is a lattice isomorphism
\begin{displaymath} 
\{\text{localizing tensor ideals of }\sfD(\QCoh
\PP^1)\}\xrightarrow{\sim}  2^{\PP^1}
\end{displaymath} where $2^{\PP^1}$ denotes the powerset of $\PP^1$
with the obvious lattice structure. The bijection is given by the
assignments
\begin{displaymath} \sfL \mapsto \{x\in \PP^1 \mid k(x) \otimes
\sfL \neq 0\}
\end{displaymath} for a localizing ideal $\sfL$ and
\begin{displaymath} \mcV\mapsto \Gamma_\mcV\sfD(\QCoh\PP^1):=\{X \in
  \sfD(\QCoh \PP^1)\mid X\otimes k(y) \cong 0 \text{ for }  y\notin
  \mcV\}
\end{displaymath} for a subset $\mcV$ of points on $\PP^1$.

This agrees with the classification of smashing subcategories in the
sense that the smashing ideals are precisely those inflated from the
compacts, i.e.\ those corresponding to specialization closed subsets
of points. Since $\PP^1$ is $1$-dimensional the only new localizing
ideals that occur are obtained by throwing the residue field of the
generic point, $k(\eta)$, into a finite localization.

Thus we have identified a sublattice consisting of a copy of $\ZZ$ and
the powerset of $\PP^1$ inside the lattice of localizing subcategories
of $\sfD(\QCoh \PP^1)$. The lattice structure extends that of the
lattice of thick subcategories of
$\sfD^\mathrm{b}(\Coh \PP^1)$ in the expected way. The naive guess is
that this is, in fact, the whole lattice. While we do not know if this
is the case, we can give an intrinsic definition of the localizations
we have stumbled into so far. This description is the goal of the next two
subsections.

\subsection{An aside on continuous pure-injectives}

In order to describe the localizations we have listed so far a word on
continuous pure-injectives is required.

\begin{defn} A pure-injective object $I$ is \emph{continuous} (or
\emph{superdecomposable}) if it has no indecomposable direct summands.
\end{defn}

We say that $\sfT$ \emph{has no continuous pure-injective objects} if
every non-zero pure-injective object has an indecomposable direct
summand or, in other words, if there are no continous pure-injectives. An equivalent condition is that every pure-injective object
is the pure-injective envelope of a coproduct of indecomposable
pure-injective objects.

\begin{prop} The category $\sfD(\QCoh \PP^1)$ has no continuous
pure-injective objects.
\end{prop}
\begin{proof} Let $A=\left[\begin{smallmatrix}k&k^2\\
0&k\end{smallmatrix}\right]$ denote the Kronecker algebra.  We use the
derived equivalence $\sfD(\QCoh \PP^1)\xrightarrow{\sim}\sfD(\Modu
A)$. Let $X$ be a pure-injective object in $\sfD(\Modu A)$.  Observe
that $X$ decomposes into a coproduct $X=\coprod_{n\in\mathbb Z}X_n$ of
complexes with cohomology concentrated in a single degree, since $A$
is a hereditary algebra. Thus we may assume that $X$ is concentrated
in degree zero and identifies with a pure-injective $A$-module. Now
the assertion follows from the description of the pure-injective
$A$-modules in \cite[Theorem~8.58]{JL1989}.
\end{proof}

\begin{cor}\label{cor:leftperp} A localizing subcategory
$\sfL\subseteq \sfD(\QCoh\PP^1)$ is cohomological if and only if there
is a collection of indecomposable pure-injective objects $(Y_i)_{i\in
I}$ in $\QCoh\PP^1$ such that $\sfL=\{X\in\sfD(\QCoh\PP^1)\mid
\Hom(X,\Sigma^jY_i)=0\text{ for all }i\in I, j\in \ZZ\}$.
\end{cor}
\begin{proof} By Proposition~\ref{prop:cohomological} being
cohomological is equivalent to being the left perpendicular of a
collection of pure-injective objects. By the last proposition
$\sfD(\QCoh \PP^1)$ has no continuous pure-injectives and so we may
replace such a collection of pure-injectives with the collection of
its indecomposable summands without changing the left perpendicular. These are all honest sheaves since $\QCoh\PP^1$ is hereditary.
\end{proof}

\subsection{Classifying cohomological localizations}

In this section we give a classification of the cohomological
localizing subcategories of $\sfD(\QCoh \PP^1)$. As we will show in
Theorem~\ref{thm:classification} they are precisely the subcategories
described in Section~\ref{sec:smashing}. Our strategy is to use
Corollary~\ref{cor:leftperp} and the classification of indecomposable
pure-injectives for $\sfD(\QCoh \PP^1)$ to compute everything
explicitly; we can compute the set of indecomposable pure-injectives
associated to each of the localizing subcategories described in
Section~\ref{sec:smashing} and show any suspension stable set of pure-injectives has the same left perpendicular as one of these.

To this end we first recall the description of the indecomposable
pure-injective objects of $\QCoh \PP^1$. Let us set up a little
notation: given a closed point $x\in \PP^1$ we can consider the
corresponding map of schemes
\begin{displaymath} i_x\colon \Spec \str_{\PP^1,x} \lto \PP^1.
\end{displaymath} We denote the maximal ideal of $\str_{\PP^1,x}$ by
$\mathfrak{m}_x$ and the residue field $\str_{\PP^1,x}/\mathfrak{m}_x$
by $k(x)$. Let $E(x)$ be the injective envelope of the residue field
$k(x)$, and $A(x)$ the $\mathfrak{m}_x$-adic completion of
$\str_{\PP^1,x}$, which is the Matlis dual of $E(x)$. Pushing these
forward along $i_x$ gives objects in $\QCoh \PP^1$ which we denote by
\begin{displaymath} \mcE(x) = {i_x}_*E(x) \quad \text{and} \quad
\mcA(x) = {i_x}_*A(x).
\end{displaymath} 

\begin{prop} The indecomposable pure-injective quasi-coherent sheaves
are given by the following disjoint classes of sheaves:
\begin{itemize}
\item the indecomposable coherent sheaves;
\item the Pr\"ufer sheaves $\mcE(x)$ for $x\in \PP^1$ closed;
\item the adic sheaves $\mcA(x)$ for $x\in \PP^1$ closed;
\item the sheaf of rational functions $k(\eta)$.
\end{itemize}
\end{prop}
\begin{proof}
  The indecomposable pure-injective quasi-coherent sheaves correspond
  to the indecomposable pure-injective modules over the Kronecker
  algebra via the derived equivalence
  $\sfD(\QCoh \PP^1)\xrightarrow{\sim}\sfD(\Modu A)$. The latter have beeen classified in \cite[Theorem~8.58]{JL1989}.
\end{proof}

\begin{rem} The Pr\"ufer sheaves and $k(\eta)$ are precisely the
indecomposable injective quasi-coherent sheaves.
\end{rem}

Having recalled the indecomposable pure-injective sheaves we now
determine how they interact with the localizations described in
Section~\ref{sec:smashing}. Let us begin by recording their supports.

\begin{lem}\label{lem:supp} We have
\begin{displaymath} \supp \mcE(x) = \{x\}, \; \supp k(\eta) =
\{\eta\}, \;\text{and}\; \supp \mcA(x) = \{x,\eta\}.
\end{displaymath}
\end{lem}
\begin{proof} All of these sheaves are pushforwards along the
inclusions of the spectra of local rings at points, and so their
supports are contained in the relevant subset $\Spec
\str_{\PP^1,x}$. Having reduced to computing the support over
$\str_{\PP^1,x}$ this is then a standard computation.
\end{proof}

As one would expect the localizations $\Loc(\str(i))$ are particularly
simple.

\begin{lem}\label{le:perp1}
  The only indecomposable pure-injective quasi-coherent sheaf in
  $\Loc(\str(i))^\perp$ is $\str(i-1)$.
\end{lem}
\begin{proof} There is a localization sequence for the compacts
\begin{displaymath} \xymatrix{ \Thick(\str(i)) \ar[r]<0.5ex>
\ar@{<-}[r]<-0.5ex> \ar[d]_-\wr & \sfD^\mathrm{b}(\Coh \PP^1)
\ar[r]<0.5ex> \ar@{<-}[r]<-0.5ex> & \Thick(\str(i-1)) \ar[d]^-\wr \\
\sfD^\mathrm{b}(k) & & \sfD^\mathrm{b}(k) }
\end{displaymath} where the adjoints exist since $\str(i)$ is
exceptional and the computation of the right orthogonal follows from
the computation of the cohomology of $\PP^1$. Applying Thomason's
localization theorem shows that
\begin{displaymath} \Loc(\str(i))^\perp = \Loc(\str(i-1)) =
\Add(\Sigma^j\str(i-1)\mid j\in \ZZ)
\end{displaymath} and the claim is then immediate.
\end{proof}

We next compute the pure-injectives lying in the right perpendicular
of the localizing ideals.

\begin{lem}\label{le:perp2}
Let $\mcV$ be a set of closed points with complement
$\mcU$. Then the indecomposable pure-injective sheaves in
$\Gamma_\mcV\sfD(\QCoh \PP^1)^\perp$ are:
\begin{itemize}
\item the indecomposable coherent sheaves supported at closed points
in $\mcU$;
\item the Pr\"ufer sheaves $\mcE(x)$ for $x\in \mcU$;
\item the adic sheaves $\mcA(x)$ for $x\in \mcU$;
\item the sheaf of rational functions $k(\eta)$.
\end{itemize}
\end{lem}
\begin{proof} By the classification of localizing ideals of
$\sfD(\QCoh \PP^1)$ we know that the category $\Gamma_\mcV\sfD(\QCoh \PP^1)^\perp$
consists of precisely those objects supported on $\mcU$. Since $\mcV$
consists of closed points we know $\mcU$ contains the generic point
$\eta$. The list is then an immediate consequence of
Lemma~\ref{lem:supp}.
\end{proof}

\begin{lem}\label{le:perp3}
  Let $\mcV$ be a subset of $\PP^1$ containing the generic point and
  denote its complement by $\mcU$. Then the indecomposable
  pure-injective sheaves in $\Gamma_\mcV\sfD(\QCoh \PP^1)^\perp$ are:
\begin{itemize}
\item the indecomposable coherent sheaves supported at closed points
in $\mcU$;
\item the adic sheaves $\mcA(x)$ for $x\in \mcU$.
\end{itemize}
\end{lem}
\begin{proof} The sheaf of rational functions $k(\eta)$ has a map
to every indecomposable injective sheaf and so no $\mcE(x)$ is
contained in the right perpendicular category (and clearly
$k(\eta)$ is not). The category $\Gamma_\mcV\sfD(\QCoh \PP^1)$
contains the torsion and adic sheaves for points in $\mcV$ so the only
indecomposable pure-injective sheaves which could lie in the right
perpendicular are those indicated; it remains to check they really
don't receive maps from objects of $\Gamma_\mcV\sfD(\QCoh \PP^1)$.

This is clear for the residue fields $k(x)$ for $x\in \mcU$, as they
cannot receive a map from any of the residue fields generating
$\Gamma_\mcV\sfD(\QCoh \PP^1)$. Since the right perpendicular is thick
it thus contains all the indecomposable coherent sheaves supported on
$\mcU$. Moreover, the right perpendicular is closed under homotopy
limits and so contains the corresponding adic sheaves $\mcA(x)$.
\end{proof}

We now know which subsets of indecomposable pure-injectives occur in
the right perpendiculars of the localizing subcategories we
understand. It's natural to ask for the minimal set giving rise to one
of these categories, as in Corollary~\ref{cor:leftperp}. Let us make
the convention that for an object $E\in \sfD(\QCoh \PP^1)$
\begin{displaymath} {}^\perp E = \{F\in \sfD(\QCoh \PP^1) \mid
\Hom(F,\Sigma^jE) = 0 \;\; \forall j\in \ZZ\}.
\end{displaymath} We can, without too much difficulty, compute all of
the left perpendiculars of the indecomposable pure-injectives.

\begin{lem}\label{lem:leftperps} 
  The left perpendicular categories to the suspension closures of the
  indecomposable pure-injectives are as follows:
\begin{enumerate}
\item ${}^\perp F = \Gamma_{\PP^1\setminus \{x\}}\sfD(\QCoh \PP^1)$
  for any $F\in\Coh\PP^1$ supported at  $x\in \PP^1$;
\item ${}^\perp\str(i) = \Loc(\str(i+1))$;
\item ${}^\perp\mcE(x) =\Gamma_{\PP^1\setminus\{x,\eta\}}\sfD(\QCoh \PP^1)$;
\item ${}^\perp \mcA(x) =\Gamma_{\PP^1\setminus \{x\}}\sfD(\QCoh \PP^1)$;
\item ${}^\perp k(\eta) = \Gamma_{\PP^1\setminus\{\eta\}}\sfD(\QCoh \PP^1)$.
\end{enumerate}
\end{lem}
\begin{proof} These are all (more or less) straightforward
computations.
\end{proof}

Knowing this it is not hard to write down minimal sets of
pure-injectives determining the ideals.

\begin{cor} Let $\mcV$ be a subset of $\PP^1$. Then
we have \[\Gamma_\mcV\sfD(\QCoh \PP^1)
= {}^\perp\{k(x)\mid x\notin \mcV\}.\]
\end{cor}

We also now have enough information to confirm that we have a complete
list of the cohomological localizing subcategories.

\begin{thm}\label{thm:classification} 
There is a lattice isomorphism
\begin{displaymath} 
\{\text{cohomological localizing subcategories
of }\sfD(\QCoh \PP^1)\} \xrightarrow{\sim}  2^{\PP^1}\amalg\ZZ,
\end{displaymath} 
where $2^{\PP^1}$ is the powerset of $\PP^1$,
with inverse defined by
\begin{displaymath} \mcV \mapsto \Gamma_\mcV\sfD(\QCoh \PP^1) \quad
  \text{and} \quad i \mapsto \Loc(\str(i)).
\end{displaymath} That is, the cohomological localizing subcategories
are precisely the localizing ideals and the $\Loc(\str(i))$ for
$i\in\ZZ$.
\end{thm}
\begin{proof} By Corollary~\ref{cor:leftperp} the cohomological
localizing subcategories are precisely the localizing subcategories
which are left perpendicular to a set of indecomposable
pure-injectives. Taking the left perpendicular of a set of
pure-injectives corresponds to intersecting the corresponding left
perpendiculars. By Lemma~\ref{lem:leftperps} we thus see that any such
localizing subcategory is of the form claimed.
\end{proof}

\begin{rem}
  Denote by $\Ind\PP^1$ the set of isomorphism classes of
  indecomposable pure-injective sheaves. The subsets of the form
  $\sfL^\perp\cap\Ind\PP^1$ for some cohomological localizing
  subcategory $\sfL$ are listed in Lemmas~\ref{le:perp1},
  \ref{le:perp2}, and \ref{le:perp3}.
\end{rem}

\section{Exotic localizations}

As noted in Section~\ref{sec:smashing} we have a classification both of ideals and of smashing localizations for $\sfD(\QCoh
\PP^1)$. Moreover, we have just shown in
Theorem~\ref{thm:classification} that together these are precisely the
cohomological localizations. It is obvious to ask if there are
non-cohomological localizations; we do not know the answer to this
question and don't hazard a guess.

In this section we at least provide some foundation for future work in
this direction by presenting some criteria to guarantee a localizing
subcategory is an ideal. This is relevant as any non-cohomological
localization could not be an ideal\textemdash{}we proved that all
ideals are cohomological. As we shall see this dramatically restricts
the possible form of a potential `exotic' localizing subcategory.

\subsection{A restriction on supports}

We begin by analysing support theoretic conditions that ensure a
localizing subcategory is an ideal. Since $\PP^1$ is $1$-dimensional
the consequences we obtain are quite strong. However, the ideas
present in the arguments should be of more general interest.

The first observation is that if the support of an object does not
contain some closed point then that object generates an ideal.

\begin{lem}\label{lem_closedideal} Let $y$ be a closed point of
$\PP^1$ and let $X\in \sfD(\QCoh \PP^1)$ be such that $y\notin \supp
X$. Then $\sfL = \Loc(X)$ is an ideal.
\end{lem}
\begin{proof} By definition we have $k(y)\otimes X\cong 0$. Since $y$
is a closed point the torsion sheaf $k(y)$ is compact, and hence
rigid, so we deduce that
\begin{displaymath} \sHom(k(y), X)\cong 0.
\end{displaymath} 
In particular, $X\in \Loc(k(y))^\perp \cong
\sfD(\QCoh \AA^1)$, where we have made an identification of
$\PP^1\setminus\{y\}$ with the affine line. Since $k(y)$ is compact
the subcategory $\Loc(k(y))^\perp$ is localizing and so
\begin{displaymath} \sfL \subseteq \Loc(k(y))^\perp \cong \sfD(\QCoh
\AA^1).
\end{displaymath} It just remains to note that every localizing
subcategory of $\sfD(\QCoh \AA^1)$ is an ideal and that
$\Loc(k(y))^\perp$ is itself an ideal, from which it is immediate that
$\sfL$ is an ideal in $\sfD(\QCoh \PP^1)$.
\end{proof}

Let $\mathcal{V} = \PP^1\setminus \{\eta\}$ denote the set of closed
points of $\PP^1$. Corresponding to this Thomason subset there is a
smashing subcategory $\Gamma_\mathcal{V}\sfD(\QCoh \PP^1)$ which comes
with a natural action of $\sfD(\QCoh \PP^1)$, in the sense of \cite{St2013}, via the corresponding
acyclization functor. Moreover, $\Gamma_\mathcal{V}\sfD(\QCoh \PP^1)$
is a tensor triangulated category in its own right, with tensor unit
$\Gamma_\mcV\str$ (which is, however, not compact).

\begin{lem}\label{lem_Vgen} The category $\Gamma_\mathcal{V}\sfD(\QCoh
\PP^1)$ is generated by its tensor unit and hence every localizing
subcategory contained in it is an ideal in it, and thus a submodule
for the $\sfD(\QCoh \PP^1)$ action. In particular, every localizing
subcategory of $\sfD(\QCoh \PP^1)$ contained in
$\Gamma_\mathcal{V}\sfD(\QCoh \PP^1)$ is an ideal of $\sfD(\QCoh
\PP^1)$.
\end{lem}
\begin{proof} The subset $\mathcal{V}$ is discrete, in the sense that
there are no specialization relations between any distinct pair of
points in it. It follows from \cite{St2017} that
$\Gamma_\mathcal{V}\sfD(\QCoh \PP^1)$ decomposes as
\begin{displaymath} \Gamma_\mathcal{V}\sfD(\QCoh \PP^1) \cong
\prod_{x\in \mathcal{V}} \Gamma_x\sfD(\QCoh \PP^1).
\end{displaymath} With respect to this decomposition the monoidal unit
$\Gamma_\mathcal{V}\str$ is just $\bigoplus_{x\in
\mathcal{V}}\Gamma_x\str$, which clearly generates. It follows that
every localizing subcategory of $\Gamma_\mathcal{V}\sfD(\QCoh \PP^1)$
is an ideal (see for instance \cite[Lemma~3.13]{St2013}) and from this the remaining statements are clear.
\end{proof}

As a particular consequence we get the following statement, which is
more in the spirit of Lemma~\ref{lem_closedideal}.

\begin{lem}\label{lem_etaideal} 
Let $X\in \sfD(\QCoh \PP^1)$ be an
object such that $\eta \notin \supp X$. Then $\Loc(X)$ is an ideal.
\end{lem}
\begin{proof} Since $\eta\notin \supp X$ we have
  $X\in \Gamma_\mathcal{V}\sfD(\QCoh \PP^1)$. Thus $\Loc(X)$ is
  contained in $\Gamma_\mathcal{V}\sfD(\QCoh \PP^1)$ and therefore an
  ideal by the previous lemma.
\end{proof}

We have shown that for any object $X$ with proper support the category
$\Loc(X)$ is an ideal. Next we will show that any localizing
subcategory containing such an object is automatically an ideal. This
requires the following technical lemma.

\begin{lem}\label{lem_quotientmonogenic} If $\sfL$ is a non-zero
localizing ideal of $\sfD(\QCoh \PP^1)$ then the quotient $\sfT =
\sfD(\QCoh \PP^1)/\sfL$ is generated by the tensor unit.
\end{lem}
\begin{proof} Since the property of being generated by the tensor unit
is preserved under taking quotients it is enough to verify the
statement when $\sfL$ has support a single point. If $\supp \sfL$ is a
closed point then we can identify $\sfT$ with the derived category of
the open complement, which is isomorphic to $\AA^1$. Having made this
observation the conclusion follows immediately.

It remains to verify the lemma in the case that $\supp \sfL =
\{\eta\}$. In this situation there is a recollement
\begin{displaymath} \xymatrix{ \Gamma_\mathcal{V}\sfD(\QCoh \PP^1)
\ar[r]<1ex> \ar@{<-}[r] \ar[r]<-1ex> & \sfD(\QCoh \PP^1) \ar[r]<1ex>
\ar@{<-}[r] \ar[r]<-1ex> & \sfL }
\end{displaymath} where, as above, $\mathcal{V}$ denotes the set of
closed points of $\PP^1$. The bottom four arrows identify
$\Gamma_\mathcal{V}\sfD(\QCoh \PP^1)$ with the quotient $\sfT$ and the
desired conclusion is given by Lemma~\ref{lem_Vgen}.
\end{proof}

Combining all of this we arrive at the following proposition.

\begin{prop}\label{prop_support} If $\sfL$ is a localizing subcategory
of $\sfD(\QCoh \PP^1)$ such that there is a non-zero $X\in \sfL$ with
$\supp X \subsetneq \PP^1$ then $\sfL$ is an ideal.
\end{prop}
\begin{proof} Let $X\in \sfL$ as in the statement of the
proposition. The object $X$ generates a non-zero localizing
subcategory $\Loc(X)\subseteq \sfL$. Since the support of $X$ is
proper and non-empty it fails to contain some point of $\PP^1$ and so,
by one of Lemma~\ref{lem_closedideal} or \ref{lem_etaideal}, it is an
ideal. We thus have a monoidal quotient functor $\sfD(\QCoh \PP^1)
\xrightarrow{\pi} \sfD(\QCoh \PP^1)/\Loc(X)$ and an induced
localizing subcategory $\sfL/\Loc(X)$ in the quotient. By
Lemma~\ref{lem_quotientmonogenic} the quotient $\sfD(\QCoh
\PP^1)/\Loc(X)$ is generated by the tensor unit and so $\sfL/\Loc(X)$
is a tensor ideal in it. But then $\sfL = \pi^{-1}(\sfL/\Loc(X))$ is
also an ideal, which completes the proof.
\end{proof}

\begin{ex} The non-ideals we know, namely the $\Loc(\str(i))$, are of
course compatible with the proposition: every object of
$\Loc(\str(i))$ is just a sum of suspensions of copies of $\str(i)$,
and these are all supported everywhere.
\end{ex}

The following interpretation is the most striking in our context.

\begin{cor}\label{cor_support} If $\sfL$ is a localizing subcategory
which is not cohomological then every non-zero object of $\sfL$ is
supported everywhere.
\end{cor}

\subsection{Twisting sheaves and avoiding compacts}

We next make a few comments concerning the interactions between
localizing subcategories and the twisting sheaves.

\begin{lem} If $\sfL$ is a localizing subcategory which is not an
ideal then
\begin{displaymath} \sfL \cap \sfL(i) = 0
\end{displaymath} for all $i\in \ZZ\setminus\{0\}$.
\end{lem}
\begin{proof} Without loss of generality we may assume $i>0$. Suppose,
for a contradiction, that $X\in \sfL\cap \sfL(i)$ is non-zero. Pick a
closed point $y$ and consider a triangle
\begin{displaymath} \str(-i) \lto \str \lto Z(y) \lto \Sigma \str(-i)
\end{displaymath} where $Z(y)$ is the cyclic torsion sheaf of length
$i$ supported at $y$. We can tensor with $X$ to get a new triangle
\begin{displaymath} X(-i) \lto X \lto X\otimes Z(y) \lto \Sigma X(-i),
\end{displaymath} where both $X$ and $X(-i)$ lie in $\sfL$ by
hypothesis. Thus, since $\sfL$ is localizing, we see that $X\otimes
Z(y)$ lies in $\sfL$. By Proposition~\ref{prop_support} we know that
$X$ is supported everywhere and so $X\otimes Z(y)\neq 0$. But on the
other hand, $X\otimes Z(y)$ is supported only at $y$ which, by the
same Proposition, implies that $\sfL$ is an ideal yielding a
contradiction.
\end{proof}

\begin{rem} The lemma implies that non-cohomological localizing
subcategories would have to come in $\ZZ$-indexed families.
\end{rem}

\begin{lem} If $\sfL$ is a localizing subcategory such that
\begin{displaymath} \Loc(\str(i)) \subsetneq \sfL
\end{displaymath} for some $i\in \ZZ$ then $\sfL = \sfD(\QCoh \PP^1)$.
\end{lem}
\begin{proof} Localizing subcategories containing $\Loc(\str(i))$ are
in bijection with localizing subcategories of $\sfD(\QCoh
\PP^1)/\Loc(\str(i))$. This quotient is just $\sfD(k)$ and so, since
we have asked for a proper containment, the result follows.
\end{proof}

We can now conclude that any non-cohomological localizing subcategory
must intersect the compact objects trivially.

\begin{prop} If $\sfL$ is a localizing subcategory which is not
cohomological then $\sfL$ contains no non-zero compact object.
\end{prop}
\begin{proof} The indecomposable compact objects are just the
indecomposable torsion sheaves at each point and the twisting
sheaves. By Lemma~\ref{lem_closedideal} we know $\sfL$ cannot contain
a torsion sheaf and by the last lemma it cannot contain a twisting
sheaf.
\end{proof}

\end{document}